\documentclass[seceqn,secthm]{elsart}%
\usepackage{amsfonts}
\usepackage{amsmath}
\usepackage{amssymb}
\usepackage{graphicx}%
\journal{Turkish Journal of Mathematics}
\date{}
\newtheorem{theorem}{Theorem}

\newtheorem{corollary}[theorem]{Corollary}

\newtheorem{definition}[theorem]{Definition}

\newtheorem{lemma}[theorem]{Lemma}

\newtheorem{remark}[theorem]{Remark}

\newenvironment{proof}[1][Proof]{\noindent\textbf{#1.} }{\ \rule{0.5em}{0.5em}}
\begin{document}
%
%TCIMACRO{\TeXButton{Begin frontmatter}{\begin{frontmatter}}}%
%BeginExpansion
\begin{frontmatter}%
%EndExpansion

%

%TCIMACRO{\TeXButton{Title}{\title{Some approximation properties of Lupa\c{s}
%$q$-analogue of Bernstein operators}}}%
%BeginExpansion
\title{Some approximation properties of Lupa\c{s}
$q$-analogue of Bernstein operators}%
%EndExpansion
%

%TCIMACRO{\TeXButton{Author}{\author{N. I.  Mahmudov and P. Sabanc{\i}gil}}}%
%BeginExpansion
\author{N. I.  Mahmudov and P. Sabanc{\i}gil}%
%EndExpansion
%

%TCIMACRO{\TeXButton{Address}{\address
%{Eastern Mediterranean University, Gazimagusa, TRNC, Mersin 10, Turkey \\ email: nazim.mahmudov@emu.edu.tr \\ pembe.sabancigil@emu.edu.tr}%
%}}%
%BeginExpansion
\address
{Eastern Mediterranean University, Gazimagusa, TRNC, Mersin 10, Turkey \\ email: nazim.mahmudov@emu.edu.tr \\ pembe.sabancigil@emu.edu.tr}%
%EndExpansion
%
\begin{abstract}%

In this paper, we discuss rates of convergence for the Lupa\c{s} $q$-analogue
of Bernstein polynomials $R_{n,q}$. We prove a quantitative variant of
Voronovskaja's theorem for $R_{n,q}$.%

%TCIMACRO{\TeXButton{End abstract}{\end{abstract}}}%
%BeginExpansion
\end{abstract}%
%EndExpansion
%

%TCIMACRO{\TeXButton{Begin keywords}{\begin{keyword}}}%
%BeginExpansion
\end{frontmatter}%
%EndExpansion

\section{Introduction}

Let $q>0.$ For any $n\in N\cup\left\{  0\right\}  $, the $q$-integer $\left[
n\right]  =\left[  n\right]  _{q}$ is defined by%
\[
\left[  n\right]  :=1+q+...+q^{n-1},\ \ \ \left[  0\right]  :=0;
\]
and the $q$-factorial $\left[  n\right]  !=\left[  n\right]  _{q}!$ by%
\[
\left[  n\right]  !:=\left[  1\right]  \left[  2\right]  ...\left[  n\right]
,\ \ \ \ \left[  0\right]  !:=1.
\]
For integers $0\leq k\leq n$, the $q$-binomial is defined by%
\[
\left[
\begin{array}
[c]{c}%
n\\
k
\end{array}
\right]  :=\frac{\left[  n\right]  !}{\left[  k\right]  !\left[  n-k\right]
!}.
\]

In the last two decades interesting generalizations of Bernstein polynomials
were proposed by Lupa\c{s} \cite{lupas}%
\[
R_{n,q}(f,x)=\sum_{k=0}^{n}f\left(  \frac{\left[  k\right]  }{\left[
n\right]  }\right)  \left[
\begin{array}
[c]{c}%
n\\
k
\end{array}
\right]  \frac{q^{\frac{k(k-1)}{2}}x^{k}(1-x)^{n-k}}{(1-x+qx)...(1-x+q^{n-1}%
x)}%
\]
and by Phillips \cite{phil1}%
\[
B_{n,q}\left(  f,x\right)  =\sum_{k=0}^{n}f\left(  \frac{\left[  k\right]
}{\left[  n\right]  }\right)  \left[
\begin{array}
[c]{c}%
n\\
k
\end{array}
\right]  x^{k}%
%TCIMACRO{\dprod _{s=0}^{n-k-1}}%
%BeginExpansion
{\displaystyle\prod_{s=0}^{n-k-1}}
%EndExpansion
\left(  1-q^{s}x\right)  .
\]
The Phillips $q$-analogue of the Bernstein polynomials ($B_{n,q}$) attracted a
lot of interest and was studied widely by a number of authors. A survey of the
obtained results and references on the subject can be found in \cite{ost2}.
The Lupa\c{s} operators ($R_{n,q}$) are less known. However, they have an
advantage of generating positive linear operators for all $q>0$, whereas
Phillips polynomials generate positive linear operators only if $q\in(0,1)$.
Lupa\c{s} \cite{lupas} investigated approximating properties of the operators
$R_{n,q}(f,x)$ with respect to the uniform norm of $C[0,1]$. In particular, he
obtained some sufficient conditions for a sequence $\{R_{n,q}(f,x)\}$ to be
approximating for any function $f\in C[0,1]$ and estimated the rate of
convergence in terms of the modulus of continuity. He also investigated
behavior of the operators $R_{n,q}(f,x)$ for convex functions. In \cite{ost3}
several results on convergence properties of the sequence $\{R_{n,q}(f,x)\}$
is presented. In particular, it is proved that the sequence $\{R_{n,q}(f,x)\}$
converges uniformly to $f(x)$ on $[0,1]$ if and only if $q_{n}\rightarrow1$.
On the other hand, for any $q>0$ fixed, $q\neq1$, the sequence $\{R_{n,q}%
(f,x)\}$ converges uniformly to $f(x)$ if and only if $f(x)=ax+b$ for some
$a,b\in R$.

In the paper, we investigate the rate of convergence for the sequence
$\{R_{n,q}(f,x)\}$ by the modulii of continuity. We discuss Voronovskaja-type
theorems for Lupa\c{s} operators for arbitrary fixed $q>0$. Moreover, for the
Voronovskaja's asymptotic formula we obtain the estimate of the remainder term.

\section{Auxiliary results}

It will be convenient to use for $x\in\lbrack0,1)$ the following
transformations
\[
v=v\left(  q,x\right)  :=\frac{qx}{1-x+qx},\ \ \ v\left(  q^{j},v\right)
=\frac{q^{j}v}{1-v+q^{j}v}.
\]
Let $0<q<1$. We set%
\begin{align*}
b_{nk}(q;x)  &  :=\left[
\begin{array}
[c]{c}%
n\\
k
\end{array}
\right]  \frac{q^{\frac{k(k-1)}{2}}x^{k}(1-x)^{n-k}}{(1-x+qx)...(1-x+q^{n-1}%
x)},\ \ x\in\left[  0,1\right]  ,\\
b_{\infty k}(q;x)  &  :=\frac{q^{\frac{k(k-1)}{2}}\left(  x/1-x\right)  ^{k}%
}{\left(  1-q\right)  ^{k}\left[  k\right]  !%
%TCIMACRO{\dprod _{j=0}^{\infty}}%
%BeginExpansion
{\displaystyle\prod_{j=0}^{\infty}}
%EndExpansion
(1+q^{j}\left(  x/1-x\right)  )},\ \ \ x\in\left[  0,1\right)  .
\end{align*}

It was proved in \cite{lupas} and \cite{ost3} that for $q\in\left(
0,1\right)  $ and $x\in\left[  0,1\right)  $,
\[
\sum_{k=0}^{n}b_{nk}(q;x)=\sum_{k=0}^{\infty}b_{\infty k}(q;x)=1.
\]
\bigskip

\begin{definition}
Lupa\c{s} \cite{lupas}.The linear operator $R_{n,q}:C\left[  0,1\right]
\rightarrow C\left[  0,1\right]  $ defined by%
\[
R_{n,q}(f,x)=\sum_{k=0}^{n}f\left(  \frac{\left[  k\right]  }{\left[
n\right]  }\right)  b_{nk}(q;x)
\]
is called the $q$-analogue of the Bernstein operator.
\end{definition}

\begin{definition}
The linear operator defined on $C\left[  0,1\right]  $ given by
\[
R_{\infty,q}\left(  f,x\right)  =\left\{
\begin{tabular}
[c]{lll}%
$%
%TCIMACRO{\dsum _{k=0}^{\infty}}%
%BeginExpansion
{\displaystyle\sum_{k=0}^{\infty}}
%EndExpansion
f\left(  1-q^{k}\right)  b_{\infty k}(q;x)$ & if & $x\in\left[  0,1\right)
,$\\
$f\left(  1\right)  $ & if & $x=1.$%
\end{tabular}
\ \ \ \ \right.
\]
is called the limit $q$-Lupa\c{s} operator.
\end{definition}

It follows directly from the definition that operators $R_{n,q}(f$,$x)$
possess the end-point interpolation property, that is,
\[
R_{n,q}(f,0)=f(0),\ \ \ \ \ R_{n,q}(f,1)=f(1)
\]
for all $q>0$ and all $n=1,2,....$

\begin{lemma}
\label{lem2}We have%
\begin{align*}
b_{nk}(q;x)  &  =\left[
\begin{array}
[c]{c}%
n\\
k
\end{array}
\right]
%TCIMACRO{\dprod _{j=0}^{k-1}}%
%BeginExpansion
{\displaystyle\prod_{j=0}^{k-1}}
%EndExpansion
v\left(  q^{j},x\right)
%TCIMACRO{\dprod _{j=0}^{n-k-1}}%
%BeginExpansion
{\displaystyle\prod_{j=0}^{n-k-1}}
%EndExpansion
\left(  1-v\left(  q^{k+j},x\right)  \right)  ,\ \ \ \ x\in\left[  0,1\right]
,\\
\ \ \ b_{\infty k}(q;x)  &  =\frac{1}{\left(  1-q\right)  ^{k}\left[
k\right]  !}%
%TCIMACRO{\dprod _{j=0}^{k-1}}%
%BeginExpansion
{\displaystyle\prod_{j=0}^{k-1}}
%EndExpansion
v\left(  q^{j},x\right)
%TCIMACRO{\dprod _{j=0}^{\infty}}%
%BeginExpansion
{\displaystyle\prod_{j=0}^{\infty}}
%EndExpansion
\left(  1-v\left(  q^{k+j},x\right)  \right)  ,\ \ \ \ \ x\in\left[
0,1\right]  .
\end{align*}

\end{lemma}

It was proved in \cite{lupas} and \cite{ost3} that $R_{n,q}\left(  f,x\right)
$, $R_{\infty,q}\left(  f,x\right)  $ reproduce linear functions and
$R_{n,q}\left(  t^{2},x\right)  $ and $R_{\infty,q}\left(  t^{2},x\right)  $
were explicitly evaluated. Using Lemma \ref{lem2} we may write formulas for
$R_{n,q}\left(  t^{2},x\right)  $ and $R_{\infty,q}\left(  t^{2},x\right)  $
in the compact form.

\begin{lemma}
We have
\begin{align*}
R_{n,q}\left(  1,x\right)   &  =1,R_{n,q}\left(  t,x\right)  =x,R_{\infty
,q}\left(  1,x\right)  =1,R_{\infty,q}\left(  t,x\right)  =x,\\
R_{n,q}\left(  t^{2},x\right)   &  =xv\left(  q,x\right)  +\frac{x\left(
1-v\left(  q,x\right)  \right)  }{\left[  n\right]  },\\
R_{\infty,q}\left(  t^{2},x\right)   &  =xv\left(  q,x\right)  +\left(
1-q\right)  x\left(  1-v\left(  q,x\right)  \right)  =x-qx\left(  1-v\left(
q,x\right)  \right)  .
\end{align*}

\end{lemma}

Now define%
\[
L_{n,q}(f,x):=R_{n,q}(f,x)-R_{\infty,q}\left(  f,x\right)  .
\]

\begin{lemma}
The following recurrence formulae hold%
\begin{align}
R_{n,q}\left(  t^{m+1},x\right)   &  =R_{n,q}\left(  t^{m},x\right)  -\left(
1-x\right)  \frac{\left[  n-1\right]  ^{m}}{\left[  n\right]  ^{m}}%
R_{n-1,q}\left(  t^{m},v\right)  ,\label{r1}\\
R_{\infty,q}\left(  t^{m+1},x\right)   &  =R_{\infty,q}\left(  t^{m},x\right)
-\left(  1-x\right)  R_{\infty,q}\left(  t^{m},v\right)  ,\label{r2}\\
L_{n,q}(t^{m+1},x)  &  =L_{n,q}(t^{m},x)+\left(  1-x\right) \nonumber\\
&  \times\left(  \left(  1-\frac{\left[  n-1\right]  ^{m}}{\left[  n\right]
^{m}}\right)  R_{\infty,q}(t^{m},v)-\frac{\left[  n-1\right]  ^{m}}{\left[
n\right]  ^{m}}L_{n-1,q}(t^{m},v)\right)  . \label{r3}%
\end{align}

\end{lemma}

\begin{proof}
First we prove (\ref{r1}). We write explicitly%
\begin{equation}
R_{n,q}(t^{m+1},x)=\sum_{k=0}^{n}\frac{\left[  k\right]  ^{m+1}}{\left[
n\right]  ^{m+1}}\left[
\begin{array}
[c]{c}%
n\\
k
\end{array}
\right]
%TCIMACRO{\dprod _{j=0}^{k-1}}%
%BeginExpansion
{\displaystyle\prod_{j=0}^{k-1}}
%EndExpansion
v\left(  q^{j},x\right)
%TCIMACRO{\dprod _{j=0}^{n-k-1}}%
%BeginExpansion
{\displaystyle\prod_{j=0}^{n-k-1}}
%EndExpansion
\left(  1-v\left(  q^{k+j},x\right)  \right)  \label{r00}%
\end{equation}
and rewrite the first two factor in the following form:%
\begin{align}
\frac{\left[  k\right]  ^{m+1}}{\left[  n\right]  ^{m+1}}\left[
\begin{array}
[c]{c}%
n\\
k
\end{array}
\right]   &  =\frac{\left[  k\right]  ^{m}}{\left[  n\right]  ^{m}}\left(
1-q^{k}\frac{\left[  n-k\right]  }{\left[  n\right]  }\right)  \left[
\begin{array}
[c]{c}%
n\\
k
\end{array}
\right] \nonumber\\
&  =\frac{\left[  k\right]  ^{m}}{\left[  n\right]  ^{m}}\left[
\begin{array}
[c]{c}%
n\\
k
\end{array}
\right]  -\frac{\left[  n-1\right]  ^{m}}{\left[  n\right]  ^{m}}\frac{\left[
k\right]  ^{m}}{\left[  n-1\right]  ^{m}}\left[
\begin{array}
[c]{c}%
n-1\\
k
\end{array}
\right]  q^{k}. \label{r01}%
\end{align}
Finally, if we substitute (\ref{r01}) in (\ref{r00}) we get (\ref{r1}):%
\begin{gather*}
R_{n,q}(t^{m+1},x)=\sum_{k=0}^{n}\frac{\left[  k\right]  ^{m}}{\left[
n\right]  ^{m}}\left[
\begin{array}
[c]{c}%
n\\
k
\end{array}
\right]
%TCIMACRO{\dprod _{j=0}^{k-1}}%
%BeginExpansion
{\displaystyle\prod_{j=0}^{k-1}}
%EndExpansion
v\left(  q^{j},x\right)
%TCIMACRO{\dprod _{j=0}^{n-k-1}}%
%BeginExpansion
{\displaystyle\prod_{j=0}^{n-k-1}}
%EndExpansion
\left(  1-v\left(  q^{k+j},x\right)  \right) \\
-\frac{\left[  n-1\right]  ^{m}}{\left[  n\right]  ^{m}}\left(  1-x\right)
\sum_{k=0}^{n-1}\frac{\left[  k\right]  ^{m}}{\left[  n-1\right]  ^{m}}\left[
\begin{array}
[c]{c}%
n-1\\
k
\end{array}
\right]
%TCIMACRO{\dprod _{j=0}^{k-1}}%
%BeginExpansion
{\displaystyle\prod_{j=0}^{k-1}}
%EndExpansion
v\left(  q^{j},v\left(  q,x\right)  \right)
%TCIMACRO{\dprod _{j=0}^{n-k-2}}%
%BeginExpansion
{\displaystyle\prod_{j=0}^{n-k-2}}
%EndExpansion
\left(  1-v\left(  q^{k+j},v\left(  q,x\right)  \right)  \right) \\
=R_{n,q}\left(  t^{m},x\right)  -\frac{\left[  n-1\right]  ^{m}}{\left[
n\right]  ^{m}}\left(  1-x\right)  R_{n-1,q}\left(  t^{m},v\left(  q,x\right)
\right)  .
\end{gather*}

Next we prove (\ref{r3})%
\begin{multline*}
L_{n,q}(t^{m+1},x)=R_{n,q}\left(  t^{m+1},x\right)  -R_{\infty,q}\left(
t^{m+1},x\right) \\
=R_{n,q}\left(  t^{m},x\right)  -\left(  1-x\right)  \frac{\left[  n-1\right]
^{m}}{\left[  n\right]  ^{m}}R_{n-1,q}\left(  t^{m},v\left(  q,x\right)
\right) \\
-R_{\infty,q}\left(  t^{m},x\right)  +\left(  1-x\right)  R_{\infty,q}\left(
t^{m},v\left(  q,x\right)  \right) \\
=L_{n,q}(t^{m},x)+\left(  1-x\right) \\
\times\left(  \left(  1-\frac{\left[  n-1\right]  ^{m}}{\left[  n\right]
^{m}}\right)  R_{\infty,q}\left(  t^{m},v\left(  q,x\right)  \right)
-\frac{\left[  n-1\right]  ^{m}}{\left[  n\right]  ^{m}}L_{n-1,q}\left(
t^{m},v\left(  q,x\right)  \right)  \right)  .
\end{multline*}
Formula (\ref{r2}) can be obtained from (\ref{r1}), by taking the limit as
$n\rightarrow\infty$.
\end{proof}

Moments $R_{n,q}\left(  t^{m},x\right)  $, $R_{\infty,q}\left(  t^{m}%
,x\right)  $ are of particular importance in the theory of approximation by
positive operators. In what follows we need explicit formulas for moments
$R_{n,q}\left(  t^{3},x\right)  $, $R_{\infty,q}\left(  t^{3},x\right)  $.

\begin{lemma}
We have
\begin{align*}
R_{n,q}\left(  t^{3},x\right)   &  =xv\left(  q,x\right)  +\frac{x\left(
1-v\left(  q,x\right)  \right)  }{\left[  n\right]  ^{2}}-\frac{\left[
n-1\right]  \left[  n-2\right]  q^{2}}{\left[  n\right]  ^{2}}x\left(
1-v\left(  q,x\right)  \right)  v\left(  q^{2},x\right)  ,\\
R_{\infty,q}\left(  t^{3},x\right)   &  =xv\left(  q,x\right)  +\left(
1-q\right)  ^{2}x\left(  1-v\left(  q,x\right)  \right)  -q^{2}x\left(
1-v\left(  q,x\right)  \right)  v\left(  q^{2},x\right)  .
\end{align*}

\end{lemma}

\begin{proof}
Note that explicit formulas for $R_{n,q}\left(  t^{m},x\right)  $,
$R_{\infty,q}\left(  t^{m},x\right)  ,$ $m=0,1,2$ were proved in \cite{lupas},
\cite{ost3}. Now we prove an explicit formula for $R_{n,q}\left(
t^{3},x\right)  ,$ since formula for $R_{\infty,q}\left(  t^{3},x\right)  $
can be obtained by taking limit as $n\rightarrow\infty$. The proof is based on
the recurrence formula (\ref{r1}). Indeed,%
\begin{align*}
R_{n,q}\left(  t^{3},x\right)   &  =R_{n,q}\left(  t^{2},x\right)  -\left(
1-x\right)  \frac{\left[  n-1\right]  ^{2}}{\left[  n\right]  ^{2}}%
R_{n-1,q}\left(  t^{2},v\right) \\
&  =xv\left(  q,x\right)  +\frac{x\left(  1-v\left(  q,x\right)  \right)
}{\left[  n\right]  }-\left(  1-x\right)  \frac{\left[  n-1\right]  ^{2}%
}{\left[  n\right]  ^{2}}v\left(  q,x\right)  v\left(  q^{2},x\right) \\
&  -\left(  1-x\right)  \frac{\left[  n-1\right]  }{\left[  n\right]  ^{2}%
}v\left(  q,x\right)  +\left(  1-x\right)  \frac{\left[  n-1\right]  }{\left[
n\right]  ^{2}}v\left(  q,x\right)  v\left(  q^{2},x\right) \\
&  =xv\left(  q,x\right)  +\frac{x\left(  1-v\left(  q,x\right)  \right)
}{\left[  n\right]  }\left(  1-\frac{q\left[  n-1\right]  }{\left[  n\right]
}\right) \\
&  -\frac{\left[  n-1\right]  }{\left[  n\right]  ^{2}}\left(  \left[
n-1\right]  -1\right)  qx\left(  1-v\left(  q,x\right)  \right)  v\left(
q^{2},x\right) \\
&  =xv\left(  q,x\right)  +\frac{x\left(  1-v\left(  q,x\right)  \right)
}{\left[  n\right]  ^{2}}-\frac{\left[  n-1\right]  \left[  n-2\right]  q^{2}%
}{\left[  n\right]  ^{2}}x\left(  1-v\left(  q,x\right)  \right)  v\left(
q^{2},x\right)  .
\end{align*}

\end{proof}

In order to prove Voronovskaja type theorem for $R_{n,q}\left(  f,x\right)  $
we also need explicit formulas and inequalities for $L_{n,q}\left(
t^{m},x\right)  $, $m=2,3,4$.

\begin{lemma}
Let $0<q<1.$ Then%
\begin{align}
L_{n,q}(t^{2},x)  &  =\frac{q^{n}}{\left[  n\right]  }x\left(  1-v\left(
q,x\right)  \right)  ,\label{lnq1}\\
L_{n,q}(t^{3},x)  &  =\frac{q^{n}}{\left[  n\right]  ^{2}}x\left(  1-v\left(
q,x\right)  \right) \label{lnq2}\\
&  \times\left[  2-q^{n}+\left[  n-1\right]  \left(  1+q\right)  v\left(
q^{2},x\right)  +\left[  n\right]  qv\left(  q^{2},x\right)  \right]
,\nonumber\\
L_{n,q}(t^{4},x)  &  =\frac{q^{n}}{\left[  n\right]  ^{2}}x\left(  1-v\left(
q,x\right)  \right)  M\left(  q,v\left(  q^{2},x\right)  ,v\left(
q^{3},x\right)  \right)  , \label{lnq3}%
\end{align}
where $M$ is a function of $\left(  q,v\left(  q^{2},x\right)  ,v\left(
q^{3},x\right)  \right)  $.
\end{lemma}

\begin{proof}
First we find a formula for $L_{n,q}\left(  t^{3},x\right)  $. To do this we
use the recurrence formula (\ref{r3}):%
\begin{align*}
&  L_{n,q}\left(  t^{3},x\right) \\
&  =L_{n,q}\left(  t^{2},x\right)  +\left(  1-x\right) \\
&  \times\left[  \left(  1-\frac{\left[  n-1\right]  ^{2}}{\left[  n\right]
^{2}}\right)  R_{\infty,q}\left(  t^{2},v\left(  q,x\right)  \right)
-\frac{\left[  n-1\right]  ^{2}}{\left[  n\right]  ^{2}}L_{n-1,q}\left(
t^{2},v\left(  q,x\right)  \right)  \right] \\
&  =\frac{q^{n}}{\left[  n\right]  }x\left(  1-v\left(  q,x\right)  \right)
+\left(  1-x\right)  \left(  1-\frac{\left[  n-1\right]  ^{2}}{\left[
n\right]  ^{2}}\right)  \left[  \left(  1-q\right)  v\left(  q,x\right)
+qv\left(  q,x\right)  v\left(  q^{2},x\right)  \right] \\
&  -\left(  1-x\right)  \frac{\left[  n-1\right]  ^{2}}{\left[  n\right]
^{2}}\frac{q^{n-1}}{\left[  n-1\right]  }v\left(  q,x\right)  \left(
1-v\left(  q^{2},x\right)  \right) \\
&  =\frac{q^{n}}{\left[  n\right]  ^{2}}x\left(  1-v\left(  q,x\right)
\right) \\
&  \times\left[  \left[  n\right]  +\left(  \frac{\left[  n\right]
^{2}-\left[  n-1\right]  ^{2}}{q^{n-1}}\right)  \left(  1-q+qv\left(
q^{2},x\right)  \right)  -\left[  n-1\right]  \left(  1-v\left(
q^{2},x\right)  \right)  \right] \\
&  =\frac{q^{n}}{\left[  n\right]  ^{2}}x\left(  1-v\left(  q,x\right)
\right) \\
&  \times\left[  \left[  n\right]  +\left(  \left[  n-1\right]  +\left[
n\right]  \right)  \left(  1-q+qv\left(  q^{2},x\right)  \right)  -\left[
n-1\right]  \left(  1-v\left(  q^{2},x\right)  \right)  \right] \\
&  =\frac{q^{n}}{\left[  n\right]  ^{2}}x\left(  1-v\left(  q,x\right)
\right) \\
&  \times\left[  \left[  n\right]  +1-q^{n-1}+1-q^{n}+\left[  n-1\right]
\left(  1+q\right)  v\left(  q^{2},x\right)  +\left[  n\right]  qv\left(
q^{2},x\right)  -\left[  n-1\right]  \right] \\
&  =\frac{q^{n}}{\left[  n\right]  ^{2}}x\left(  1-v\left(  q,x\right)
\right)  \left[  2-q^{n}+\left[  n-1\right]  \left(  1+q\right)  v\left(
q^{2},x\right)  +\left[  n\right]  qv\left(  q^{2},x\right)  \right]  .
\end{align*}
The proof of the equation (\ref{lnq3}) is also elementary, but tedious and
complicated. Just notice that we use recurrence formula for $L_{n,q}\left(
t^{4},x\right)  $ and clearly each term of the formula contains $\frac{q^{n}%
}{\left[  n\right]  ^{2}}x\left(  1-v\left(  q,x\right)  \right)  $.
\end{proof}

\begin{lemma}
We have%
\begin{align}
L_{n,q}\left(  \left(  t-x\right)  ^{2},x\right)   &  =\frac{q^{n}}{\left[
n\right]  }x\left(  1-v\left(  q,x\right)  \right)  ,\label{m1}\\
L_{n,q}\left(  \left(  t-x\right)  ^{3},x\right)   &  =\frac{q^{n}}{\left[
n\right]  ^{2}}x\left(  1-v\left(  q,x\right)  \right) \label{m2}\\
&  \times\left[  2-q^{n}+\left[  n-1\right]  \left(  1+q\right)  v\left(
q^{2},x\right)  +\left[  n\right]  qv\left(  q^{2},x\right)  -3\left[
n\right]  x\right]  ,\nonumber\\
L_{n,q}\left(  \left(  t-x\right)  ^{4},x\right)   &  \leq K_{1}\frac{q^{n}%
}{\left[  n\right]  ^{2}}x\left(  1-v\left(  q,x\right)  \right)  , \label{m3}%
\end{align}
where $K_{1}$ is a positive constant.
\end{lemma}

\begin{proof}
Proofs of (\ref{m2}) and (\ref{m3}) are based on (\ref{lnq2}), (\ref{lnq3})
and on the following identities.%
\begin{align*}
L_{n,q}\left(  \left(  t-x\right)  ^{3},x\right)   &  =L_{n,q}\left(
t^{3},x\right)  -3xL_{n,q}\left(  \left(  t-x\right)  ^{2},x\right)  ,\\
L_{n,q}\left(  \left(  t-x\right)  ^{4},x\right)   &  =L_{n,q}\left(
t^{4},x\right)  -4xL_{n,q}\left(  \left(  t-x\right)  ^{3},x\right)
-6x^{2}L_{n,q}\left(  \left(  t-x\right)  ^{2},x\right)  .
\end{align*}

\end{proof}

\section{Convergence properties}

For $f\in C[0,1],$ $t>0$, the modulus of continuity $\omega(f,t)$ and the
second modulus of smoothness $\omega_{2}(f,t)$ of $f$ are defined by%
\begin{align*}
\omega(f,t)  &  =\sup_{\left\vert x-y\right\vert \leq t}\left\vert
f(x)-f(y)\right\vert ,\ \ \\
\ \omega_{2}(f,t)  &  =\sup_{0\leq h\leq t}\sup_{0\leq x\leq1-2h}\left\vert
f(x+2h)-2f(x+h)+f(x)\right\vert .
\end{align*}

In \cite{ost3}, it is proved that $b_{nk}(q;x)\rightarrow b_{\infty k}(q;x)$
uniformly in $x\in\left[  0,1\right)  $ as $n\rightarrow\infty$. In the next
lemma we show that this convergence is uniform on $\left(  0,q_{0}\right]
\times\left[  0,1\right)  $ and give some estimates for $\left\vert
b_{nk}(q;x)-b_{\infty k}(q;x)\right\vert $.

\begin{lemma}
\label{lem1}Let $0<q\leq q_{0}<1$, $k\geq0,$ $n\geq1.$

\begin{enumerate}
\item[(i)] For any $\varepsilon>0$ there exists $M>0$ such that%
\[
\left\vert b_{nk}(q;x)-b_{\infty k}(q;x)\right\vert \leq b_{nk}(q;x)M\left(
\varepsilon\right)  \frac{\left(  q_{0}+\varepsilon\right)  ^{n}}{1-\left(
q_{0}+\varepsilon\right)  }+b_{\infty k}(q;x)\frac{q_{0}^{n-k+1}}{1-q_{0}}%
\]
for all $\left(  q,x\right)  \in\left(  0,q_{0}\right]  \times\left[
0,1\right)  $. In particular, $b_{nk}(q;x)$ converges to $b_{\infty k}(q;x)$
uniformly in $\left(  q,x\right)  \in\left(  0,q_{0}\right]  \times\left[
0,1\right)  $.

\item[(ii)] For any $x\in\left[  0,1\right)  $ we have%
\[
\left\vert b_{nk}(q;x)-b_{\infty k}(q;x)\right\vert \leq b_{nk}(q;x)\frac
{x}{1-x}\frac{q^{n}}{1-q}+b_{\infty k}(q;x)\frac{q^{n-k+1}}{1-q}.
\]
In particular, $b_{nk}(q;x)$ converges to $b_{\infty k}(q;x)$ uniformly in
$\left(  q,x\right)  \in\left(  0,q_{0}\right]  \times\left[  0,a\right]  $,
$0<a<1$.
\end{enumerate}
\end{lemma}

\begin{proof}
We only prove part (i), since the proof of (ii) is similar to that of (i).
Standard computations show that%
\begin{align}
&  \left\vert b_{nk}(q;x)-b_{\infty k}(q;x)\right\vert \nonumber\\
&  =\left\vert \left[
\begin{array}
[c]{c}%
n\\
k
\end{array}
\right]
%TCIMACRO{\dprod _{j=0}^{k-1}}%
%BeginExpansion
{\displaystyle\prod_{j=0}^{k-1}}
%EndExpansion
v\left(  q^{j},x\right)
%TCIMACRO{\dprod _{j=0}^{n-k-1}}%
%BeginExpansion
{\displaystyle\prod_{j=0}^{n-k-1}}
%EndExpansion
\left(  1-v\left(  q^{k+j},x\right)  \right)  \right. \nonumber\\
&  \left.  -\frac{1}{(1-q)^{k}\left[  k\right]  !}%
%TCIMACRO{\dprod _{j=0}^{k-1}}%
%BeginExpansion
{\displaystyle\prod_{j=0}^{k-1}}
%EndExpansion
v\left(  q^{j},x\right)
%TCIMACRO{\dprod \limits_{j=0}^{\infty}}%
%BeginExpansion
{\displaystyle\prod\limits_{j=0}^{\infty}}
%EndExpansion
\left(  1-v\left(  q^{k+j},x\right)  \right)  \right\vert \nonumber\\
&  =\left\vert \left[
\begin{array}
[c]{c}%
n\\
k
\end{array}
\right]
%TCIMACRO{\dprod _{j=0}^{k-1}}%
%BeginExpansion
{\displaystyle\prod_{j=0}^{k-1}}
%EndExpansion
v\left(  q^{j},x\right)  \left(
%TCIMACRO{\dprod _{j=0}^{n-k-1}}%
%BeginExpansion
{\displaystyle\prod_{j=0}^{n-k-1}}
%EndExpansion
\left(  1-v\left(  q^{k+j},x\right)  \right)  -%
%TCIMACRO{\dprod \limits_{j=0}^{\infty}}%
%BeginExpansion
{\displaystyle\prod\limits_{j=0}^{\infty}}
%EndExpansion
\left(  1-v\left(  q^{k+j},x\right)  \right)  \right)  \right. \nonumber\\
&  +\left.
%TCIMACRO{\dprod _{j=0}^{k-1}}%
%BeginExpansion
{\displaystyle\prod_{j=0}^{k-1}}
%EndExpansion
v\left(  q^{j},x\right)
%TCIMACRO{\dprod \limits_{j=0}^{\infty}}%
%BeginExpansion
{\displaystyle\prod\limits_{j=0}^{\infty}}
%EndExpansion
\left(  1-v\left(  q^{k+j},x\right)  \right)  \left(  \left[
\begin{array}
[c]{c}%
n\\
k
\end{array}
\right]  -\frac{1}{(1-q)^{k}\left[  k\right]  !}\right)  \right\vert
\nonumber\\
&  \leq b_{nk}(q;x)\left\vert 1-%
%TCIMACRO{\dprod \limits_{j=n}^{\infty}}%
%BeginExpansion
{\displaystyle\prod\limits_{j=n}^{\infty}}
%EndExpansion
\left(  1-v\left(  q^{j},x\right)  \right)  \right\vert +b_{\infty
k}(q;x)\left\vert
%TCIMACRO{\dprod \limits_{j=n-k+1}^{n}}%
%BeginExpansion
{\displaystyle\prod\limits_{j=n-k+1}^{n}}
%EndExpansion
(1-q^{j})-1\right\vert . \label{a1}%
\end{align}
Now using the inequality%
\[
1-%
%TCIMACRO{\dprod \limits_{j=1}^{k}}%
%BeginExpansion
{\displaystyle\prod\limits_{j=1}^{k}}
%EndExpansion
(1-a_{j})\leq\sum_{j=1}^{k}a_{j},\text{ }(a_{1},a_{2},...,a_{k}\in(0,1),\text{
}k=1,2,...,\infty),
\]
we get from (\ref{a1}) that
\begin{equation}
\left\vert b_{nk}(q;x)-b_{\infty k}(q;x)\right\vert \leq b_{nk}(q;x)\sum
_{j=n}^{\infty}v\left(  q^{j},x\right)  +b_{\infty k}(q;x)\sum_{j=n-k+1}%
^{n}q^{j}. \label{a2}%
\end{equation}

On the other hand, $\lim_{j\rightarrow\infty}\dfrac{v\left(  q^{j+1},x\right)
}{v\left(  q^{j},x\right)  }=q<1$ and observe for any $\varepsilon>0$ such
that $q_{0}+\varepsilon<1$ there exists $n^{\ast}\in\mathbb{N}$ such that%
\[
\dfrac{v\left(  q^{j+1},x\right)  }{v\left(  q^{j},x\right)  }<q_{0}%
+\varepsilon=\frac{\left(  q_{0}+\varepsilon\right)  ^{j+1}}{\left(
q_{0}+\varepsilon\right)  ^{j}}%
\]
for all $j>n^{\ast}$. Hence, the sequence $v\left(  q^{j},x\right)  /\left(
q_{0}+\varepsilon\right)  ^{j}$ is decreasing for large $j$ and thus uniformly
bounded in $\left(  q,x\right)  \in\left(  0,q_{0}\right]  \times\left[
0,1\right)  $ by
\[
M\left(  \varepsilon\right)  =\max\left\{  \frac{v\left(  q^{n^{\ast}%
+1},x\right)  }{\left(  q_{0}+\varepsilon\right)  ^{n^{\ast}+1}}%
,\frac{v\left(  q^{n^{\ast}},x\right)  }{\left(  q_{0}+\varepsilon\right)
^{n^{\ast}}},...,\frac{v\left(  q,x\right)  }{q_{0}+\varepsilon}\right\}  .
\]
So, for such $M\left(  \varepsilon\right)  >0$ we have $\left\vert v\left(
q^{j},x\right)  \right\vert \leq M\left(  \varepsilon\right)  \left(
q_{0}+\varepsilon\right)  ^{j}$ for all $j=1,2,...$ and $\left(  q,x\right)
\in\left(  0,q_{0}\right]  \times\left[  0,1\right)  $.

Now from (\ref{a2}) we get the desired inequality%
\[
\left\vert b_{nk}(q;x)-b_{\infty k}(q;x)\right\vert \leq b_{nk}(q;x)M\left(
\varepsilon\right)  \frac{\left(  q_{0}+\varepsilon\right)  ^{n}}{1-\left(
q_{0}+\varepsilon\right)  }+b_{\infty k}(q;x)\frac{q_{0}^{n-k+1}}{1-q_{0}}.
\]

\end{proof}

Before proving the main results notice that the following theorem proved in
\cite{ost3} will allow us to reduce the case $q\in(1,\infty)$ to the case
$q\in(0,1)$.

\begin{theorem}
\label{ost} Let $f\in C\left[  0,1\right]  ,$ $g\left(  x\right)  :=f\left(
1-x\right)  $. Then for any $q>1,$%
\[
R_{n,q}\left(  f,x\right)  =R_{n,\frac{1}{q}}\left(  g,1-x\right)
\ \ \ \ \text{and\ \ \ }R_{\infty,q}\left(  f,x\right)  =R_{\infty,\frac{1}%
{q}}\left(  g,1-x\right)  .
\]

\end{theorem}

Using Lemma \ref{lem1} we prove the following quantitative result for the rate
of local convergence of $R_{n,q}\left(  f,x\right)  $ in terms of the first
modulus of continuity.

\begin{theorem}
\label{est}Let $0<q<1$ and $f\in C\left[  0,1\right]  $. Then for all $0\leq
x<1$ we have%
\[
\left\vert R_{n,q}\left(  f,x\right)  -R_{\infty,q}\left(  f,x\right)
\right\vert \leq\frac{2}{1-q}\frac{1}{1-x}\omega\left(  f,q^{n}\right)  .
\]

\end{theorem}

\begin{proof}
Consider%
\[
\Delta\left(  x\right)  :=R_{n,q}(f,x)-R_{\infty,q}(f,x)=\sum_{k=0}%
^{n}f\left(  \frac{\left[  k\right]  }{\left[  n\right]  }\right)
b_{nk}(q;x)-\sum_{k=0}^{\infty}f\left(  1-q^{k}\right)  b_{\infty k}(q;x).
\]
Since $R_{n,q}(f,x)$ and $R_{\infty,q}(f,x)$ possess the end point
interpolation property $\Delta\left(  0\right)  =\Delta\left(  1\right)  =0$.
For all $x\in\left(  0,1\right)  $ we rewrite $\Delta$ in the following form%
\begin{multline*}
\Delta\left(  x\right)  =\sum_{k=0}^{n}\left[  f\left(  \frac{\left[
k\right]  }{\left[  n\right]  }\right)  -f\left(  1-q^{k}\right)  \right]
b_{nk}(q;x)\\
+\sum_{k=0}^{n}\left[  f\left(  1-q^{k}\right)  -f\left(  1\right)  \right]
\left(  b_{nk}(q;x)-b_{\infty k}(q;x)\right) \\
-\sum_{k=n+1}^{\infty}\left[  f\left(  1-q^{k}\right)  -f\left(  1\right)
\right]  b_{\infty k}(q;x)=:I_{1}+I_{2}+I_{3}.
\end{multline*}
We start with estimation of $I_{1}$ and $I_{3}$. Since
\begin{align*}
0  &  \leq\frac{\left[  k\right]  }{\left[  n\right]  }-\left(  1-q^{k}%
\right)  =\frac{1-q^{k}}{1-q^{n}}-\left(  1-q^{k}\right)  =q^{n}\frac{1-q^{k}%
}{1-q^{n}}\leq q^{n},\\
0  &  \leq1-\left(  1-q^{k}\right)  =q^{k}\leq q^{n},\ \ \ \ k>n,
\end{align*}
we get%
\begin{align}
\left\vert I_{1}\right\vert  &  \leq\omega\left(  f,q^{n}\right)  \sum
_{k=0}^{n}b_{nk}(q;x)=\omega\left(  f,q^{n}\right)  ,\label{q1}\\
\left\vert I_{3}\right\vert  &  \leq\omega\left(  f,q^{n}\right)  \sum
_{k=n+1}^{\infty}b_{\infty k}(q;x)\leq\omega\left(  f,q^{n}\right)  .
\label{q2}%
\end{align}
Finally we estimate $I_{2}$. Using the property of the modulus of continuity
\[
\omega\left(  f,\lambda t\right)  \leq\left(  1+\lambda\right)  \omega\left(
f,t\right)  ,\ \ \ \lambda>0
\]
and Lemma \ref{lem1} we get
\begin{align}
\left\vert I_{2}\right\vert  &  \leq\sum_{k=0}^{n}\omega\left(  f,q^{k}%
\right)  \left\vert b_{nk}(q;x)-b_{\infty k}(q;x)\right\vert \nonumber\\
&  \leq\omega\left(  f,q^{n}\right)  \sum_{k=0}^{n}\left(  1+q^{k-n}\right)
\left\vert b_{nk}(q;x)-b_{\infty k}(q;x)\right\vert \nonumber\\
&  \leq2\omega\left(  f,q^{n}\right)  \frac{1}{q^{n}}\sum_{k=0}^{n}%
q^{k}\left\vert b_{nk}(q;x)-b_{\infty k}(q;x)\right\vert \nonumber\\
&  \leq2\omega\left(  f,q^{n}\right)  \frac{1}{q^{n}}\sum_{k=0}^{n}%
q^{k}\left(  b_{nk}(q;x)\frac{x}{1-x}\frac{q^{n}}{1-q}+b_{\infty k}%
(q;x)\frac{q^{n-k+1}}{1-q}\right) \nonumber\\
&  \leq\frac{2}{1-q}\left(  \frac{x}{1-x}+1\right)  \omega\left(
f,q^{n}\right)  =\frac{2}{1-q}\frac{1}{1-x}\omega\left(  f,q^{n}\right)  .
\label{q3}%
\end{align}
From (\ref{q1}), (\ref{q2}), and (\ref{q3}), we conclude the desired estimation.
\end{proof}

\begin{corollary}
Let $q>1$ and $f\in C\left[  0,1\right]  $. Then for all $0<x\leq1$ we have%
\[
\left\vert R_{n,q}\left(  f,x\right)  -R_{\infty,q}\left(  f,x\right)
\right\vert \leq\frac{2q}{q-1}\frac{1}{x}\omega\left(  g,q^{-n}\right)  .
\]

\end{corollary}

\begin{proof}
Proof follows from Theorems \ref{est} and \ref{ost}.
\end{proof}

Next corollary gives quantitative result for the rate of uniform convergence
of $R_{n,q}\left(  f,x\right)  $ in $C\left[  0,a\right]  $ and $C\left[
a,1\right]  $, $0<a<1$.

\begin{corollary}
Let $f\in C\left[  0,1\right]  ,$ $0<a<1.$

\begin{enumerate}
\item If $0<q<1$, then%
\[
\left\Vert R_{n,q}\left(  f\right)  -R_{\infty,q}\left(  f\right)  \right\Vert
_{C\left[  0,a\right]  }\leq\frac{2}{1-q}\frac{1}{1-a}\omega\left(
f,q^{n}\right)  .
\]

\item If $q>1$, then
\[
\left\Vert R_{n,q}\left(  f\right)  -R_{\infty,q}\left(  f\right)  \right\Vert
_{C\left[  a,1\right]  }\leq\frac{2q}{q-1}\frac{1}{a}\omega\left(
g,q^{-n}\right)  .
\]

\end{enumerate}
\end{corollary}

In order to prove the estimation in terms of the second modulus of continuity
we need the following theorem proved in \cite{wang4}.

\begin{theorem}
\label{wang}\cite{wang4}\textbf{ }\textit{Let }$\left\{  T_{n}\right\}
$\textit{ be a sequence of positive linear operators on }$C\left[  0,1\right]
$\textit{ satisfying the following conditions:}

\begin{enumerate}
\item[(A)] \textit{the sequence }$\left\{  T_{n}\left(  t^{2}\right)  \left(
x\right)  \right\}  $\textit{ converges uniformly on }$\left[  0,1\right]  ;$

\item[(B)] \textit{the sequence }$\left\{  T_{n}\left(  f\right)  \left(
x\right)  \right\}  $\textit{ is nonincreasing in }$n$\textit{ for any convex
function }$f$\textit{ and any }$x\in\left[  0,1\right]  .$\textit{ }%
\newline\textit{Then there exists an operator }$T_{\infty}$\textit{ on
}$C\left[  0,1\right]  $\textit{ such that}%
\[
T_{n}\left(  f\right)  \left(  x\right)  \rightarrow T_{\infty}\left(
f\right)  \left(  x\right)
\]
\textit{as }$n\rightarrow\infty$\textit{ uniformly on }$\left[  0,1\right]
$\textit{. In addition, the following estimation holds:}%
\[
\left\vert T_{n}\left(  f\right)  \left(  x\right)  -T_{\infty}\left(
f\right)  \left(  x\right)  \right\vert \leq C\omega_{2}\left(  f;\sqrt
{\lambda_{n}\left(  x\right)  }\right)  ,
\]
\textit{where }$\omega_{2}$\textit{ is the second modulus of smoothness,
}$\lambda_{n}\left(  x\right)  =\left\vert T_{n}\left(  t^{2}\right)  \left(
x\right)  -T_{\infty}\left(  t^{2}\right)  \left(  x\right)  \right\vert
$\textit{, and }$C$\textit{ is a constant depending only on }$T_{1}\left(
1\right)  $\textit{.}
\end{enumerate}
\end{theorem}

\begin{theorem}
\label{thm2} Let $0<q<1$. Then
\begin{equation}
\left\Vert R_{n,q}\left(  f\right)  -R_{\infty,q}\left(  f\right)  \right\Vert
\leq c\omega_{2}\left(  f,\sqrt{q^{n}}\right)  . \label{l1}%
\end{equation}
Moreover,%
\begin{equation}
\sup_{0<q\leq1}\left\Vert R_{n,q}\left(  f\right)  -R_{\infty,q}\left(
f\right)  \right\Vert \leq c\omega_{2}\left(  f,n^{-1/2}\right)  , \label{l2}%
\end{equation}
where $c$ is a constant.
\end{theorem}

\begin{proof}
From \cite{lupas}, we know that the $q$-Bernstein operators satisfy condition
(B) of Theorem \ref{wang}. On the other hand%
\begin{equation}
0\leq R_{n,q}\left(  t^{2},x\right)  -R_{\infty,q}\left(  t^{2},x\right)
=\frac{q^{n}}{\left[  n\right]  }x\left(  1-v\left(  q,x\right)  \right)  \leq
q^{n}\frac{x\left(  1-x\right)  }{1-x+qx}\leq q^{n} \label{r5}%
\end{equation}
and%
\[
\sup_{0<q<1}\left\vert R_{n,q}\left(  t^{2},x\right)  -R_{\infty,q}\left(
t^{2},x\right)  \right\vert =\sup_{0<q<1}\frac{q^{n}}{\left[  n\right]  }%
\frac{x\left(  1-x\right)  }{1-x+qx}=\frac{x\left(  1-x\right)  }{n}.
\]
Since%
\[
\left\vert R_{n,1}\left(  t^{2},x\right)  -x^{2}\right\vert =\frac{x\left(
1-x\right)  }{n},
\]
we conclude that%
\begin{equation}
\sup_{0<q\leq1}\left\vert R_{n,q}\left(  t^{2},x\right)  -R_{\infty,q}\left(
t^{2},x\right)  \right\vert \leq\frac{x\left(  1-x\right)  }{n}\leq\frac{1}%
{n}. \label{r6}%
\end{equation}
Theorem follows from (\ref{r6}), (\ref{r5}) and Theorem \ref{wang}.
\end{proof}

\begin{theorem}
Let $q>1$. Then
\[
\left\Vert R_{n,q}\left(  f\right)  -R_{\infty,q}\left(  f\right)  \right\Vert
\leq c\omega_{2}\left(  g,\sqrt{q^{-n}}\right)  .
\]
Moreover,%
\begin{equation}
\sup_{1\leq q<\infty}\left\Vert R_{n,q}\left(  f\right)  -R_{\infty,q}\left(
f\right)  \right\Vert \leq c\omega_{2}\left(  g,n^{-1/2}\right)  , \label{l3}%
\end{equation}
where $c$ is a constant.
\end{theorem}

\begin{proof}
The proof is similar to that of Theorem \ref{thm2}.
\end{proof}

\begin{remark}
From (\ref{l2}) and (\ref{l3}), we conclude that the rate of convergence
$\left\Vert R_{n,q}\left(  f\right)  -R_{\infty,q}\left(  f\right)
\right\Vert $ can be dominated by $c\omega_{2}\left(  f,n^{-1/2}\right)  $
uniformly with respect to $q\neq1$.
\end{remark}

\begin{remark}
We may observe here that for $f\left(  x\right)  =x^{2}$, we have%
\[
\left\Vert R_{n,q}\left(  f\right)  -R_{\infty,q}\left(  f\right)  \right\Vert
\asymp q^{n}\asymp\omega_{2}\left(  f,\sqrt{q^{n}}\right)  ,\ \ 0<q<1,
\]
where $A(n)\asymp B(n)$ means that $A(n)\ll B(n)$ and $A(n)\gg B(n)$, and
$A(n)\ll B(n)$ means that there exists a positive constant $c$ independent of
$n$ such that $A(n)\leq cB(n)$. Hence the estimate (\ref{l1}) is sharp in the
following sense: the sequence $q^{n}$ in (\ref{l1}) cannot be replaced by any
other sequence decreasing to zero more rapidly as $n\rightarrow\infty$.
\end{remark}

\section{Voronovskaja type results}

\begin{theorem}
\label{thm1}Let $0<q<1,$ $f\in C^{2}[0,1].$ Then there exists a positive
absolute constant $K$ such that
\begin{align}
&  \left\vert \frac{\left[  n\right]  }{q^{n}}\left(  R_{n,q}(f,x)-R_{\infty
,q}(f,x)\right)  -\frac{f^{\prime\prime}(x)}{2}x\left(  1-v\left(  q,x\right)
\right)  \right\vert \label{s1}\\
&  \leq Kx\left(  1-v\left(  q,x\right)  \right)  \omega(f^{\prime\prime
},\left[  n\right]  ^{-\frac{1}{2}}).\nonumber
\end{align}

\end{theorem}

\begin{proof}
Let $x\in(0,1)$ be fixed. We set%
\[
g(t)=f(t)-\left(  f(x)+f^{\prime}(x)(t-x)+\frac{f^{\prime\prime}(x)}%
{2}(t-x)^{2}\right)  .
\]
It is known that (see \cite{lupas}) if the function $h$ is convex on $[0,1],$
then%
\[
R_{n,q}(h,x)\geq R_{n+1,q}(h,x)\geq...\geq R_{\infty,q}(h,x),
\]
and therefore,%
\[
L_{n,q}(h,x):=R_{n,q}(h,x)-R_{\infty,q}(h,x)\geq0.
\]
Thus $L_{n,q}$ is positive on the set of convex functions on $\left[
0,1\right]  $. But in general $L_{n,q}$ is not positive on $C\left[
0,1\right]  $.

Simple calculation gives%
\[
L_{n,q}\left(  g,x\right)  =\left(  R_{n,q}(f,x)-R_{\infty,q}(f,x)\right)
-\frac{q^{n}}{\left[  n\right]  }\frac{f^{\prime\prime}(x)}{2}x\left(
1-v\left(  q,x\right)  \right)  .
\]
In order to prove the theorem, we need to estimate $L_{n,q}\left(  g,x\right)
$. To do this, it is enough to choose a function $S\left(  t\right)  $ such
that the functions $S\left(  t\right)  \pm g\left(  t\right)  $ are convex on
$\left[  0,1\right]  $. Then $L_{n,q}\left(  S\pm g,x\right)  \geq0$, and
therefore,%
\[
\left\vert L_{n,q}\left(  g(t),x\right)  \right\vert \leq L_{n,q}\left(
S(t),x\right)  .
\]
So the first thing to do is to find such function $S\left(  t\right)  $. Using
the well-known inequality $\omega(f,\lambda\delta)\leq(1+\lambda^{2}%
)\omega(f,\delta)$ $(\lambda,\delta$ $>0),$ we get%
\begin{align*}
\left\vert g^{\prime\prime}(t)\right\vert  &  =\left\vert f^{\prime\prime
}(t)-f^{\prime\prime}(x)\right\vert \leq\omega(f^{\prime\prime},\left\vert
t-x\right\vert )\\
&  =\omega\left(  f^{\prime\prime},\frac{1}{\left[  n\right]  ^{\frac{1}{2}}%
}\left[  n\right]  ^{\frac{1}{2}}\left\vert t-x\right\vert \right)  \leq
\omega\left(  f^{\prime\prime},\frac{1}{\left[  n\right]  ^{\frac{1}{2}}%
}\right)  \left(  \left(  1+\left[  n\right]  (t-x)^{2}\right)  \right)  .
\end{align*}
Define $S\left(  t\right)  =\omega\left(  f^{\prime\prime},\left[  n\right]
^{-\frac{1}{2}}\right)  \left[  \frac{1}{2}(t-x)^{2}+\frac{1}{12}\left[
n\right]  (t-x)^{4}\right]  $. Then%
\[
\left\vert g^{\prime\prime}(t)\right\vert \leq\frac{1}{6}\omega\left(
f^{\prime\prime},\left[  n\right]  ^{-\frac{1}{2}}\right)  \left(
3(t-x)^{2}+\frac{1}{2}\left[  n\right]  (t-x)^{4}\right)  _{t}^{\prime\prime
}=S^{\prime\prime}(t)
\]
Hence the functions $S(t)\pm g(t)$ are convex on $[0,1],$ and therefore,%
\[
\left\vert L_{n,q}\left(  g(t),x\right)  \right\vert \leq L_{n,q}\left(
S(t),x\right)  ,
\]
and%
\[
L_{n,q}\left(  S(t),x\right)  =\frac{1}{6}\omega\left(  f^{\prime\prime
},\left[  n\right]  ^{-\frac{1}{2}}\right)  \left(  \frac{3q^{n}}{\left[
n\right]  }x\left(  1-v\left(  q,x\right)  \right)  +\frac{1}{2}\left[
n\right]  L_{n,q}\left(  (t-x)^{4},x\right)  \right)  .
\]
Since by the formula (\ref{m3})
\begin{equation}
L_{n,q}\left(  (t-x)^{4},x\right)  \leq K_{1}\frac{q^{n}}{\left[  n\right]
^{2}}x\left(  1-v\left(  q,x\right)  \right)  \label{s2}%
\end{equation}
we have%
\begin{equation}
L_{n,q}\left(  S(t),x\right)  \leq\frac{1}{6}\omega\left(  f^{\prime\prime
},\left[  n\right]  ^{-\frac{1}{2}}\right)  \left(  3\frac{q^{n}}{\left[
n\right]  }x\left(  1-v\left(  q,x\right)  \right)  +\frac{1}{2}\left[
n\right]  K_{1}\frac{q^{n}}{\left[  n\right]  ^{2}}x\left(  1-v\left(
q,x\right)  \right)  \right)  . \label{s3}%
\end{equation}
By (\ref{s2}) and (\ref{s3}), we obtain (\ref{s1}). Theorem is proved.
\end{proof}

\begin{corollary}
Let $q>1,$ $f\in C^{2}[0,1].$ Then there exists a positive absolute constant
$K$ such that
\begin{align*}
&  \left\vert q^{n}\left[  n\right]  _{\frac{1}{q}}\left(  R_{n,q}%
(f,x)-R_{\infty,q}(f,x)\right)  -\frac{f^{\prime\prime}(1-x)}{2}v\left(
q,x\right)  \left(  1-x\right)  \right\vert \\
&  \leq Kv\left(  q,x\right)  \left(  1-x\right)  \omega(g^{\prime\prime
},\left[  n\right]  _{\frac{1}{q}}^{-\frac{1}{2}}).
\end{align*}

\end{corollary}

\begin{corollary}
If $f\in C^{2}[0,1]$ and $q_{n}\rightarrow1$ as $n\rightarrow\infty$, then%
\begin{align}
\lim_{q_{n}\uparrow1}\left[  n\right]  _{q_{n}}\left(  R_{n,q_{n}%
}(f,x)-f\left(  x\right)  \right)   &  =\frac{f^{\prime\prime}(x)}{2}x\left(
1-x\right)  ,\label{v1}\\
\lim_{q_{n}\downarrow1}\left[  n\right]  _{\frac{1}{q_{n}}}\left(  R_{n,q_{n}%
}(f,x)-f\left(  x\right)  \right)   &  =\frac{f^{\prime\prime}(1-x)}%
{2}x\left(  1-x\right) \nonumber
\end{align}
uniformly on $\left[  0,1\right]  $.
\end{corollary}

\begin{remark}
When $q_{n}\equiv1,$ (\ref{v1}) reduces to the classical Voronovskaja's
formula. For the function $f(t)=t^{2},$ the exact equality%
\begin{align*}
\frac{\left[  n\right]  _{q}}{q^{n}}\left(  R_{n,q}(t^{2},x)-R_{\infty
,q}(t^{2},x)\right)   &  =x\left(  1-v\left(  q,x\right)  \right)
,\ \ \ \ \ \ \ 0<q<1,\\
q^{n}\left[  n\right]  _{\frac{1}{q}}\left(  R_{n,q}(t^{2},x)-R_{\infty
,q}(t^{2},x)\right)   &  =v\left(  q,x\right)  \left(  1-x\right)
,\ \ \ \ \ q>1,
\end{align*}
takes place without passing to the limit, but in contrast to the Phillips
$q$-analogue of the Bernstein polynomials the right hand side depends on $q$.
In contrast to the classical Bernstein polynomials and Phillips $q$-analogue
of the Bernstein polynomials the exact equality%
\[
\left[  n\right]  \left(  B_{n,q}(t^{2},x)-x^{2}\right)  =\left(
x^{2}\right)  ^{\prime\prime}x\left(  1-x\right)  /2
\]
does not hold for the Lupa\c{s} $q$-analogue of the Bernstein polynomials.
\end{remark}

\end{document}